\documentclass[10pt]{article}
\usepackage{indentfirst, latexsym, bm, amssymb,amsmath}

\begin{document}

\newtheorem{theorem}{Theorem}[section]
\newtheorem{corollary}[theorem]{Corollary}
\newtheorem{definition}[theorem]{Definition}
\newtheorem{conjecture}[theorem]{Conjecture}
\newtheorem{question}[theorem]{Question}
\newtheorem{lemma}[theorem]{Lemma}
\newtheorem{proposition}[theorem]{Proposition}
\newtheorem{example}[theorem]{Example}
\newtheorem{problem}[theorem]{Problem}
\newenvironment{proof}{\noindent {\bf
Proof.}}{\rule{3mm}{3mm}\par\medskip}
\newcommand{\remark}{\medskip\par\noindent {\bf Remark.~~}}
\newcommand{\pp}{{\it p.}}
\newcommand{\de}{\em}

\newcommand{\JEC}{{\it Europ. J. Combinatorics},  }
\newcommand{\JCTB}{{\it J. Combin. Theory Ser. B.}, }
\newcommand{\JCT}{{\it J. Combin. Theory}, }
\newcommand{\JGT}{{\it J. Graph Theory}, }
\newcommand{\ComHung}{{\it Combinatorica}, }
\newcommand{\DM}{{\it Discrete Math.}, }
\newcommand{\ARS}{{\it Ars Combin.}, }
\newcommand{\SIAMDM}{{\it SIAM J. Discrete Math.}, }
\newcommand{\SIAMADM}{{\it SIAM J. Algebraic Discrete Methods}, }
\newcommand{\SIAMC}{{\it SIAM J. Comput.}, }
\newcommand{\ConAMS}{{\it Contemp. Math. AMS}, }
\newcommand{\TransAMS}{{\it Trans. Amer. Math. Soc.}, }
\newcommand{\AnDM}{{\it Ann. Discrete Math.}, }
\newcommand{\NBS}{{\it J. Res. Nat. Bur. Standards} {\rm B}, }
\newcommand{\ConNum}{{\it Congr. Numer.}, }
\newcommand{\CJM}{{\it Canad. J. Math.}, }
\newcommand{\JLMS}{{\it J. London Math. Soc.}, }
\newcommand{\PLMS}{{\it Proc. London Math. Soc.}, }
\newcommand{\PAMS}{{\it Proc. Amer. Math. Soc.}, }
\newcommand{\JCMCC}{{\it J. Combin. Math. Combin. Comput.}, }
\newcommand{\GC}{{\it Graphs Combin.}, }

\title{On the spectral radius of simple digraphs with prescribed number of  arcs\thanks{
 This work is supported by National Natural Science
Foundation of China (No.11271256), Innovation Program of Shanghai Municipal Education Commission (No.14ZZ016) and Specialized Research Fund for the Doctoral Program of Higher Education (No.20130073110075).
\newline \indent $^{\dagger}$Correspondent author:
Xiao-Dong Zhang (Email: xiaodong@sjtu.edu.cn)}}
\author{  Ya-Lei Jin  and Xiao-Dong Zhang$^{\dagger}$   \\
{\small Department of Mathematics, and Ministry of Education }\\
{\small Key Laboratory of Scientific and Engineering Computing, }\\
{\small Shanghai Jiao Tong University} \\
{\small  800 Dongchuan road, Shanghai, 200240, P.R. China}\\
}

\date{}
\maketitle
 \begin{abstract}
  This paper presents a sharp upper bound for the spectral radius of simple digraphs with described number of arcs.
   Further, the extremal graphs which attain the maximum spectral radius  among all simple digraphs with fixed arcs  are investigated. In particular,  we characterize all extremal simple digraphs with the maximum spectral radius among all simple digraphs with   arcs number $e=2{k\choose 2}+t$ and  $k>4t^4+4$.   \end{abstract}


\vskip 0.5cm

\section{Introduction}
 Let $D=(V,E)$(or $(V(D),E(D))$) be a simple digraph (i.e., no loops and no multiarcs) with vertex set $V$ and arc set $E$, where $|V|=n$ and $|E|=e$. The {\it loop} is the arc which starts and ends at a same vertex $v$. The {\it multiarcs} are the arcs which start at a same vertex $v_i$ and end at a same vertex $v_j$, where $v_i\neq v_j$. The {\it adjacent matrix} of $D$ is $A(D)=(a_{ij})$ or $A$ for short,
where $a_{ij}=1$ if there is an arc from $v_i$ to $v_j$, $0$ otherwise. Then by Perron-Frobenius theorem, there is an eigenvalue
 $\rho(D)$ which is the largest modula value of all eigenvalues of $A(D)$. Moreover, $\rho(D)$ is called the {\it spectral radius} of $D$. The {\it $n$-complete simple digraph}  is the simple digraph $\overleftrightarrow{K_n}$ in which every pair of vertices is an arc, while  {\it the $n$-complete  digraph} with loops is the digraph $\overleftrightarrow{K_n^0}$ in which every pair of vertices is an arc including a loop at each vertex. Hence $\overleftrightarrow{K_n}$ has $n(n-1)$ arcs while $\overleftrightarrow{K_n^0}$ has $n^2$ arcs. The {\it clique number} of a simple digraph $D$, denoted by $w(D)$, is the maximal integer $k$ such that the  $k$-complete simple digraph is a subgraph of $D$.

 In 1985,  Brualdi and Hoffman \cite{BrH} firstly investigated  the maximum
spectral radius  for a digraph (maybe have loops but no multiarcs) with $e$ arcs.
 \begin{theorem}\cite{Br}\label{brualdi1985}
 Let $D$ be a digraph (loops are allowed but no multiarcs) with $e=m^2$ or $e=m^2+1$.
 Then
 \begin{equation}\label{th1-1}
\rho(D)\le m
\end{equation}
with equality if and only if apart from isolated vertices, $D$ is a complete digraph $\overleftrightarrow{K_m^0}$ of order $m$ for $e=m^2$; $D$ is a complete digraph $\overleftrightarrow{K_m^0}$ of order $m$ with one
additional arc for $e=m^2+1$ and $m\ge 3$.
 \end{theorem}
 Later,  Friedland \cite{Fr}  characterized the extremal digraphs for several classes of the digraphs (maybe have loops but no multiarcs) with some other arcs numbers $e$.
\begin{theorem}\cite{Fr}\label{frienland1985}
Let $D$ be a digraph (maybe have loops but no multiarcs) with $e$ arcs.

(1). If $e=m^2+l$ and $1\le l\le 2m$, then

\begin{equation}\label{th2-1}
\rho(D)\le \frac{m+\sqrt{m^2+2l}}{2},
\end{equation}
with equality if and only if
$l = 2m$ and, apart from isolated vertices, $D$ is obtained from the complete
digraph $\overleftrightarrow{K_{m+1}^0}$ of order $m + 1$ by removing a loop at one vertex.

(2). If $e=m^2 +2m-3$ and $m\ge 3$, then
\begin{equation}\label{th2-2}
\rho(D)\le \frac{m-1+\sqrt{m^2+6m-7}}{2},
\end{equation}
with equality if and only if
$D$ is obtained from a complete digraph $\overleftrightarrow{K_{m+1}^0}$ of order $m +1$
by removing a complete digraph $ \overleftrightarrow{K_2^0}$ of order 2.

(3). If $l\ge 2$, there exists a constant $C_l$ such that if $m\ge C_l$,
 the maximum value of the spectral radius of  digraphs (loops are allowed but no multiarcs) with $e=m^2+l$  arcs
 can be achieved by the spectral radius of a digraph obtained from a complete digraph $\overleftrightarrow{K_m^0}$ of order $m$ by including a new vertex
$u$ and arcs in both directions joining $u$ and $\lfloor\frac{l}{2}\rfloor$
 vertices of $\overleftrightarrow{K_m^0}$, and, if $l$ is odd, an arc in either
direction joining $u$ and an additional vertex of $\overleftrightarrow{K_m^0}$.
  \end{theorem}
On the other hand, Snellman \cite{Sn} proved the following result which in some sense is complementary to that of (3) in Theorem~\ref{frienland1985}.
\begin{theorem}\cite{Sn}\label{snellman}
Let $s\neq 4$ be a positive integer. Then there exists a constant $C_s$ such that if
$m> C_s,$
 the maximum value of the spectral radius of  digraphs (loops are allowed but no multiarcs) with $e=(m+1)^2-s$  arcs and $m+1$ vertices
  can be attained by a digraph obtained from $\overleftrightarrow{K_{m+1}^0}$ by removing the loop at a vertex
  $w$ together with $\lfloor\frac{s}{2}\rfloor$  pairs
of arcs (in both directions) between $w$ and other vertices, and in the case that $s$ is odd, one additional arc from $w$ to another vertex.
 \end{theorem}
 However, until now, the problem of characterizing all extremal graphs with the maximum spectral radius of all digraphs (loops  are allowed but no multiarcs) with fixed arcs number $e$ is not completely solved.  For the spectral radius of digraphs,  several upper bounds for digraphs in terms of digraph parameters, such as degree, clique number etc, can be found in \cite{butler2014,Ch, DrL,LiD,LiS,LiSW,Zhang2002}.
 For more  results on the spectra of the digraphs, you can refer to the excellent survey \cite{Br}.
 In the above theorems, they always considered the spectral radius of all digraphs with loops and the fixed number of arcs. It is  natural to ask  what is the  spectral radius of all simple digraphs (no loops, no multiarcs)  with the fixed number of arcs.     In this paper, we  mainly consider the following problem:
\begin{problem} Let $e$ be an integer and $\mathcal{D}(e)$ be the set all simple digraphs with the fixed number of arcs $e$. Denote by
\begin{equation}
\rho(e)=\max\{\rho(D)|D\in \mathcal{D}(e)\}.
\end{equation}
Determine the value  $\rho(e)$ in terms  of $e$  and  characterize all extremal simple digraphs with $\rho(e)$.
\end{problem}
This problem depends only on the number of arcs $e$ but not on the number of vertices
$n$. It is easy to see that $\rho(e)$ is an increasing function on $e$. Moreover, if $e=k(k-1)+t$ with $0\le t\le 2k-1$, let $ D^\#$ be the simple digraphs of order $k+1$ which are obtained from    $\overleftrightarrow{K_k}$ by adding a new vertex $v$ and  $\lfloor\frac{t}{2}\rfloor$ arcs in both directions joining $v$ and $\lfloor\frac{t}{2}\rfloor$  vertices of $\overleftrightarrow{K_k}$, and, if $t$ is odd, an arc in either
direction joining $v$ and an additional vertex of $\overleftrightarrow{K_k}$. Then
\begin{displaymath}
A(D^\#)^T=\left(\begin{matrix}J_k-I_k&\alpha_{\lfloor\frac{t}{2}\rfloor}\cr \alpha_{\lceil\frac{t}{2}\rceil}^T&0\end{matrix}\right) ~or ~
A(D^\#)=\left(\begin{matrix}J_k-I_k&\alpha_{\lfloor\frac{t}{2}\rfloor}\cr \alpha_{\lceil\frac{t}{2}\rceil}^T&0\end{matrix}\right),
\end{displaymath}
where $\alpha_{\lfloor\frac{t}{2}\rfloor}$ is a $k-$dimensional vector whose first $\lfloor\frac{t}{2}\rfloor$ components are $1$, the others are $0$. It is easy to see that
$ \rho(D^{\#}) $ is the largest positive  root of the equation
 $$\lambda^3-(k-2)\lambda^2-(k+\lfloor\frac{t}{2}\rfloor-1)\lambda+\lfloor\frac{t}{2}\rfloor
 (k-\lceil\frac{t}{2}\rceil-1)=0.$$
The main result of this paper may be stated as follows.
\begin{theorem}\label{mainth1}
Let $e=k(k-1)+t$ be a positive integer with $2\le t\le 2k-1$. If  $k>4t^4+4$, then $\rho(e)=\rho(D^{\#})$.
  Moreover, if $\rho(D)=\rho(e)$
 for $D\in \mathcal{D}(e)$, then apart from isolated vertices, $D=D^{\#}$. In other words,
 for any $D\in \mathcal{D}(e)$, $\rho(D)\le \rho(D^{\#})$ with equality if and only if apart from isolated vertices, $D=D^{\#}$.
 \end{theorem}
 In addition, for some values $e$, we characterize all extremal simple digraphs.
\begin{theorem}\label{mainth2}
(1). If $e=k(k-1),$  then $\rho(e)=k-1$. Moreover, $$\rho(D)=\rho(e)=k-1$$ for $D\in \mathcal{D}(e)$ if and only if $D$ is,  apart from isolated vertices,  complete simple digraph  $\overleftrightarrow{K_k}$.

(2). If $e=k(k-1)+1,$  then $\rho(e)=k-1$. Moreover, if $k>2$, then $$\rho(D)=\rho(e)=k-1$$ for $D\in \mathcal{D}(e)$ if and only if $D$ is,  apart from isolated vertices, the complete simple digraph $\overleftrightarrow{K_k}$ with one additional arc. If $k=2$, then $$\rho(D)=k-1$$ for $D\in \mathcal{D}(e)$ if and only if $D$ is, apart from isolated vertices, oriented triangle or  the complete simple digraph $\overleftrightarrow{K_2}$ with one additional arc.

(3). If $e=k(k-1)+2k-2$, then $\rho(e)=\frac{k-2+\sqrt{(k-2)^2+8(k-1)}}{2}$. Moreover, $$ \rho(D)= \rho(e)=\frac{k-2+\sqrt{(k-2)^2+8(k-1)}}{2}$$  for  $D\in \mathcal{D}(e)$ if and only if $D$ is,  apart from isolated vertices,  complete simple
digraph  $\overleftrightarrow{K_{k+1}}$  by removing complete simple digraph $\overleftrightarrow{K_2}$.

(4). If $e=k(k-1)+2k-1$, then $\rho(e)=\frac{k-1+\sqrt{(k-1)^2+4(k-1)}}{2}$. Moreover, $$ \rho(D)=\rho(e)=\frac{k-1+\sqrt{(k-1)^2+4(k-1)}}{2}$$ for $D\in \mathcal{D}(e)$ if and only if $D$ is,  apart from isolated vertices,  complete simple digraph  $\overleftrightarrow{K_{k+1}}$  by removing one arc.
\end{theorem}

The rest part of this paper is organized as follows. In section 2, some known useful and preliminary results are provided. In Sections  3 and 4, we present the proofs of Theorems~\ref{mainth1} and \ref{mainth2}, respectively.

\section{Preliminaries}

An $n\times n$ nonnegative matrix   $A=(a_{ij})\ge 0
 $ is  called {\it irreducible} if for all $1\le i, j\le n$ there is an integer $k$ such that $(A^k)_{ij}>0$. It is well known that $A(D)$ is
irreducible if and only if $D$ is strongly connected. If $A$ is irreducible, by Perron-Frobenius theorem, there exist two positive vectors $u,v$ such that $$\sum_{i=1}^nu_i=\sum_{i=1}^nv_i=1,~Au= \rho(A)u,~A^Tv=\rho(A)v,$$
and $u,v^T$ are called the {\it Perron} and {\it left Perron vector} of $A$. Let
$$\nu(A)=\sqrt{\rho(AA^T)},$$
where $A$ is a matrix (maybe not a square matrix), then $\nu(\cdot)$ is a {\it matrix norm} on $M_n(R)$ which is the set of all $n\times n$ real matrices (see p.~295-296\cite{HoJ}). Thus $\nu(\cdot)$ is a sub-multiplicative norm, i.e., for $A,~B\in M_n(R)$,
$$\nu(AB)\le\nu(A)\nu(B).$$
Furthermore, $\nu(\cdot)$ is an induced matrix norm induced by Euclidean vector norm $||\cdot||$. Then, for $x\in R^n$,
$$||Ax||^2=x^TA^TAx\le \nu(A)^2||x||^2.$$
The definition of $\nu(\cdot)$ implies that
$$\nu(A)^2=\nu(A^T)^2\le tr(AA^T).$$
Where $tr(AA^T)$ denotes the trace of $AA^T$. If $A$ is an $m\times n$ matrix and $B$ is an $n\times m$ matrix, then by using $\nu(\cdot)$ is a sub-multiplicative norm, we have
$$\nu(AB)\le \nu(A)\nu(B),$$
since $\nu(AB)=\nu([A,0]\cdot[B^T,0]^T)\le\nu([A,0])\nu([B^T,0]^T)= \nu(A)\nu(B)$ for $m\ge n$, where $[A,0],~[B^T,0]^T$ are $m\times m$ matrices. Similarly for $m< n$.
Friedland proved the following results in \cite{Fr}
\begin{theorem}\cite{Fr}\label{Fr1}
Let $\nu(\cdot)$ be the spectral norm of the matrices and
\begin{equation}\label{fr-1}
A=\left(\begin{matrix}0&A_{12}\cr A_{21}&0\end{matrix}\right),~B=\left(\begin{matrix}B_{11}&0\cr 0&0\end{matrix}\right).
\end{equation}
Then $\rho(A+B)\le \frac{\nu(B_{11})+(\nu(B_{11})^2+4\nu(A_{12}A_{21}))^{1/2}}{2}$.
\end{theorem}
\begin{theorem}\cite{Fr}\label{Fr2}
Let $A,B$ be nonnegative matrices with the form $(\ref{fr-1})$ and
$$B_{11}=\beta uv^T-\gamma I,~ u,v\ge 0,~v^Tu=1,~\beta >\gamma >0.$$
Then $\rho(A+B)$ is the unique positive solution of $r$ of
$$\sum_{i=0}^{\infty}\beta \frac{v^T(A_{12}A_{21})^iu}{r^i(r+\gamma)^{i+1}}=1.$$
\end{theorem}
It follows from Theorem~\ref{Fr2} that
\begin{corollary}\label{FrC}
Let $A$ and $B$ be  nonnegative matrices with the form $(\ref{fr-1})$. If
$B_1=J_k-I_k$, then $\rho(A+B)$ is the unique positive solution of $r$ of
$$\sum_{i=0}^{\infty} \frac{\mathbf{1}^T(A_{12}A_{21})^i\mathbf{1}}{r^i(r+1)^{i+1}}=1,$$
where $\mathbf{1}$ is all ones vector.
\end{corollary}
In order to study the spectral radius of digraphs in $\mathcal{D}(e)$, we need more notations. Let $\mathcal{D^*}(e)$ be the set of  all simple strongly connected digraphs with $e$ arcs. In addition, let $\mathcal{D^{**}}(e)$ be the set of all simple  strongly connected digraphs whose vertex set $\{v_1, \cdots, v_n\}$ can be arranged  such that the following two conditions holds.
(i).  If $(v_i, v_j)\in E(D)$ with $1\le i<j\le n$,  then $(v_i, v_l)\in E(D)$ for $l=1, \cdots, j$ and $l\neq i$; (ii). $N^+(v_i)\setminus\{v_j\}\supseteq N^+(v_j)\setminus\{v_i\}$ for $1\le i<j\le n$, where $N^+(v_i)=\{v|~(v_i,v)\in E(D)\}$. It is easy to see that $D^{\#}\in \mathcal{D^{**}}(e), ~e=k(k-1)+t,~0\le t\le 2k-1,~t\neq 1$. In subsequence content, without loss of generality, for any $D\in \mathcal{D^{**}}(e)$, we always assume that  the vertices $\{v_1, \cdots, v_n\}$ of $D$ are arranged to satisfy  the above two conditions. Let $V_1\subset V(D)$, the {\it induced digraph} by $V_1$ in $D$ is denoted by $D[V_1]$, which is the graph with vertex set $V_1$ and edge set $E_1=\{(u,v)\in E(D)|u,~v\in V_1\}$. Moreover, it is easy to see that the  following proposition holds.
\begin{proposition}\label{PL1}
Let $D\in \mathcal{D}^{**}(e)$ with vertex set $\{v_1, \cdots, v_n\}$. If $w(D)=w$, then $D[\{v_1,v_2,...,v_w\}]\\=\overleftrightarrow{K_w}$, $(v_{i},v_{j}),(v_{j},v_{i})\notin E(D)$ with $i,~j>w$,  and $(v_w,v_j)\notin E(D)$ for $j=w+1,\cdots, n$ or $(v_j,v_w)\notin E(D)$ for $j=w+1,\cdots n$. In other words,
\begin{equation}\label{Equation3}
A(D)=\left(\begin{matrix}J_w-I_w&A_{12}\cr A_{21}&0\end{matrix}\right),
\end{equation}
where the last row of $A_{12}$ is a zero vector or the last column of $A_{21}$ is a zero vector.
\end{proposition}
 Without loss of generality, we always assume  that the last column of $A_{21}$ is a zero vector.  Further, for the fixed integer $e$, denote by
$$\rho^*(e)=\max\{\rho(D)|\ D\in \mathcal{D^*}(e)\},$$
$$\rho^{**}(e)=\max\{\rho(D)|\ D\in \mathcal{D^{**}}(e)\}.$$
Now we discuss the relationship among $\rho(e), \rho^*(e)$ and $\rho^{**}(e)$.

\begin{proposition}\label{T1}
Let $e=k(k-1)+t$ with $0\le t\le 2k-1.$ If $t\neq 1$, then $\rho(e)=\rho^*(e)$, in other words, the digraphs having the maximum spectral radius in $\mathcal{D}(e)$, apart from isolated vertices, are strongly connected.
\end{proposition}
\begin{proof} Clearly,  $\rho^*(e)\le \rho(e)$, because of $\mathcal{D^*}(e)\subsetneq\mathcal{D}(e)$. Conversely, let $D\in \mathcal{D}$  be a digraph  without isolated vertices such that $\rho(D)=\rho(e)$. Now we claim that $D$ has to be strongly connected. In fact, if $D$ is not strongly connected, then $A(D)$ is reducible. Hence there exists a permutation matrix $P$ such that
$$PA(D)P^T=\left(\begin{matrix}A(D_1)&0\\ A_{21}&A(D_2) \end{matrix}\right),$$
where $D_1$ is strongly connected with $e_1<e$ arcs and $\rho(D)=\rho(D_1)$,  and $\left(A_{21}, A(D_2)\right)$ contains at least one 1.  it is observed that
 the spectral radius of  digraphs is nondecreasing with respect to adding an arc, i.e., $\rho(e)\le \rho(e+1)$. If $D_1$ is not the simple complete digraph $\overleftrightarrow{K_w}$, then
 $\rho(D)=\rho(D_1)<\rho(D_1+a)\le \rho(e_1+1)\le\rho(e)$, where $a$ is an arc.
 it is a contradiction. If $D_1$ is  the simple complete digraph $\overleftrightarrow{K_w}$, then $\left(A_{21}, A(D_2)\right)$ contains at least $e-k(k-1)=t>1$ arcs.  Hence let $D_3$ be strongly connected digraph obtained from $D_1$ by adding a new vertex and joining bidirected arcs. Then  $\rho(e)=\rho(D)=\rho(D_1)<\rho(D_3)\le \rho(e_1+2)\le \rho(e)$. It is a contradiction. Hence $D$ is strongly connected and the assertion holds.
 \end{proof}
\begin{proposition}\label{T2}
Let $e=k(k-1)+t,~0\le t\le 2k-1.$  If $t\neq 1$, then $\rho(e)=\rho^*(e)=\rho^{**}(e)$.
\end{proposition}
\begin{proof}
It follows from Proposition~\ref{T1} that it is sufficient to prove the assertion if $\rho(D)=\rho(e)$ for $D\in \mathcal{D^*}(e)$,  then there exists a simple digraph $D'\in \mathcal{D^{**}}(e)$ such that $\rho(D')=\rho(e)$.  Let $x=(x_1, \cdots, x_n)^T$ be Perron vector of $A(D)$ with $x_1\ge x_2\cdots\ge x_n$.  If there exist $1\le i\le j$ and $1\le l<j$ with $l\neq i$ such that $(v_i, v_j)\in E(D)$ and $(v_i, v_l)\notin E(D)$, then let $D_1$ be a simple digraph with $e$ arcs obtained from $D$ by deleting an arc $(v_i, v_j)$ and adding an arc $(v_i, v_l)$. Clearly, $A(D_1)x\ge \rho(e)x$ which implies that
$ \rho(D_1)\ge\rho(e).$ Hence $ \rho(e)=\rho(D_1)=\rho(e)$ and $x$ is also Perron vector of $A(D_1)$. Further, by Proposition~\ref{T1}, $D_1$ is strongly connected.
By repeating this process,  there exists a simple strongly connected digraph $D_2$ such that $x$ is the Perron vector of $A(D_2) $ corresponding to $\rho(e)$, and satisfies the following proposition that if $(v_i, v_j)\in E(D)$ with $1\le i<j\le n$,  then $(v_i, v_l)\in E(D)$ for $l=1, \cdots, j$ and $l\neq i$. Further, we claim   $N^+(v_i)\setminus\{v_j\}\supseteq N^+(v_j)\setminus\{v_i\}$ for $1\le i<j\le n$. In fact, if $N^+(v_i)\setminus \{v_j\}\nsupseteq N^+(v_j)\setminus \{v_i\}$, then by (1), we have $d(v_i)<d(v_j)$. We consider the  following three cases:

{\bf Case 1}. $d(v_j)\le i$, then $0\le \rho(e)(x_i-x_j)=-\sum\limits_{t=d(v_i)+1}^{d(v_j)}x_t<0$, which  is a contradiction.

{\bf Case 2}. $i<d(v_j)\le j$, then $0\le \rho(e)(x_i-x_j)\le -\sum\limits_{t=d(v_i)+1}^{d(v_j)}x_t<0$, which  is a contradiction.

{\bf Case 3}. $j<d(v_j)$, then $ \rho(e)(x_i-x_j)\le -\sum\limits_{t=\max\{d(v_i),i\}+1}^{d(v_j)}x_t+x_j-x_i$. Hence $0\le (\rho(e)+1)(x_i-x_j)$

$\le -\sum\limits_{t=\max\{d(v_i),i\}+1}^{d(v_j)}x_t<0$, which is also a contradiction.\\
Hence $D_2\in \mathcal{D^{**}}(e)$ and $\rho(D_2)=\rho(e)$. This completes the proof.
\end{proof}
\begin{corollary}\label{EL6}
Let $~e=k(k-1)+t$ with $0\le t\le 2k-1$ and $t\neq 1$.
 If $D^\#$ is the only simple digraph having the maximum spectral radius in the set $\mathcal{D^{**}}(e)$, then apart from isolated vertices, $D^\#$ is the only simple digraph having the maximum spectral radius in the set $\mathcal{D}(e)$.
  \end{corollary}
\begin{proof}
Let $D$ be any simple digraph with $e$ arcs and no isolated vertices which has the maximum spectral radius in the set $\mathcal{D}(e)$, i.e., $\rho(D)=\rho(e)$. By Proposition~\ref{T1}, $D$ has to be strongly connected. By the proof of Proposition~\ref{T2}, there exists a strongly connected digraph $D_1\in \mathcal{D^{**}}(e)$ such that $\rho(D_1)=\rho(D)$ and $x$ is the Perron vector of $A(D)$ and $A(D_1)$ corresponding to eigenvalue $\rho(e)$. By the condition of Corollary~\ref{EL6}, $D_1=D^{\#}$. Hence $x=(x_1, \cdots, x_n)^T$  is the Perron vector of $A(D^{\#})$ which implies
$x_1=x_2=\cdots=x_{t/2}>x_{1+t/2}=\cdots=x_k>x_{k+1}$. Therefore it follows from $A(D)x=\rho(e)x$ that $v_i(1\le i\le t/2)$ is adjacent to all the other vertices, $v_j(t/2 +1\le j\le k)$ is adjacent to all other vertices except $v_{k+1}$, and $v_{k+1}$ is adjacent to $v_i,~1\le i\le t/2$  for even $t$. Then  $D=D^\#$. If $t$ is odd number, by the same method, it is easy to see that $D=D^\#$.  So the assertion holds.
\end{proof}

\section{Proof of the theorem~\ref{mainth1}}

In order to present the proof of Theorem~\ref{mainth1}, we begin to give several upper bounds for the spectral radius of digraphs in the set $\mathcal{D}^{**}(e)$, which is interesting in its own right.

\begin{lemma}\label{Lu}
Let $e=2{k\choose 2}+t$ with $0\le t\le 2k-1$ and $t\neq 1.$ If $D\in \mathcal{D}^{**}(e)$, then $\rho(D)\le \frac{w-1+\sqrt{(w-1)^2+2(e-w(w-1))}}{2}\le \frac{k-1+\sqrt{(k-1)^2+2t}}{2}, $ where $w$ is the clique number of $D$.
\end{lemma}
\begin{proof}
Since $D\in \mathcal{D}^{**}(e)$,  we assume that $A(D)$ has the form $(\ref{Equation3})$ by Proposition~\ref{PL1}. Hence by Theorem~\ref{Fr1},\\
\begin{eqnarray*}
\rho(D)&\le& \frac{\rho(J_w-I_w)+\sqrt{\rho(J_w-I_w)^2+4\nu(A_{12}A_{21})}}{2}\\
&\le&\frac{w-1+\sqrt{(w-1)^2+4\nu(A_{12})\nu(A_{21})}}{2}\\
&\le&\frac{w-1+\sqrt{(w-1)^2+4\sqrt{|E(A_{12})||E(A_{21})|}}}{2}\\
&\le&\frac{w-1+\sqrt{(w-1)^2+2(|E(A_{12})|+|E(A_{21})|)}}{2}\\
&=&\frac{w-1+\sqrt{(w-1)^2+2(e-w(w-1))}}{2}.
\end{eqnarray*}
Let $2f(w)=w-1+\sqrt{(w-1)^2+2(e-w(w-1))}$. Then $2f'(w)=1-\frac{w}{\sqrt{(w-1)^2+2(e-w(w-1))}}>0$ for $1\le w\le k-1$.
 On the other hand,  $f(k-1)\le f(k)=\frac{k-1+\sqrt{(k-1)^2+2t}}{2}$. This completes the proof of this lemma.
\end{proof}

In particular, for $D\in {\mathcal{D}}^{**}(e)$  and $w(D)=k$, we characterize all extremal digraphs with the maximum spectral radius.

\begin{lemma}\label{Lw}
Let $e=2{k\choose 2}+t$ with $0\le t\le 2k-1$. If  $D\in \mathcal{D}^{**}(e)$ and  $w(D)=k>2$, then $\rho(D)\le \rho(D^\#)$ with equality  if and only if $D=D^\#$.
\end{lemma}
\begin{proof}
Since $D\in \mathcal{D}^{**}(e)$, we assume that $A(D)$ has the following form.
 $$A(D)=\left(\begin{matrix}J_k-I_k&A_{12}&A_{13}\cr A_{21}&0&0\cr A_{31}&0&0\end{matrix}\right)_{n\times n},$$
where $A_{12}, A_{21}^T$ are $k\times 1$ matrices. Moreover, $y=(y_1, \cdots, y_n)^T$ is an eigenvector of $A(D)$ corresponding to $\rho(D)$  with $y_1\ge y_2\ge \cdots\ge y_n>0$.  If $n=|V(D)|>k+1$, then
 denote by  $|E(A_{ij})|$  the number of $1$ in $A_{ij}$. It is easy to see that $p=:|E(A_{12})|+|E(A_{13})|\ge t/2$ or $q=:|E(A_{21})|+|E(A_{31})\ge t/2$, say $q\ge t/2$. Clearly, $t-\min\{q,k-1\}\le k-1$.  Let
   $$B=\left(\begin{matrix}J_k-I_k& B_{12}&0\\ B_{21}&0&0\\ 0&0&0\end{matrix}\right)_{n\times n},$$
  where $B_{12}$ is a $k-$dimensional column vector whose first $t-\min\{q,k-1\}$ components are 1, and 0 otherwise, $B_{21}$ is a $k-$dimensional row vector whose first $l=\min\{q, k\}$ components are 1, and 0 otherwise. Moreover, let $D_1$ be a simple digraph whose adjacency matrix is the $(k+1)\times (k+1)$ principal submatrix of $B$. Then $\rho(D_1)=\rho(B)$ and $D_1\in {\mathcal{D}}(e)$.  Let $x=(x_1, \cdots, x_{k+1}, 0,\cdots, 0)^T$ be the positive eigenvector of $B^T$ corresponding to $\rho(B)$. Then $x_1=\cdots=x_q\ge x_{q+1}\ge \cdots\ge x_{k+1}>0$. If $B_{21}-A_{21}\neq 0 $, then $x_{k+1}(B_{21}-A_{21})(y_1, \cdots, y_k)^T>0$. Further, since the last $k-p$ rows of $(A_{12}, A_{13})$ are zero,  we have
  $(x_1, \cdots, x_k)[(B_{12}-A_{12})y_{k+1}-A_{13}(y_{k+2}, \cdots, y_n)^T]
  \ge (x_1, \cdots, x_k)[(B_{12}-A_{12})y_{k+1}-A_{13}(1, \cdots, 1)^T y_{k+2}]\ge x_p(y_{k+1}-y_{k+2})|E(A_{13})|\ge 0$. Therefore,
\begin{eqnarray*}
&&(\rho(D_1)-\rho(D))x^Ty=x^T(B-A(D))y\\
&=&x^T\left(\begin{matrix}0& B_{12}-A_{12}&-A_{13}\cr B_{21}-A_{21}& 0&0\cr -A_{31}&0&0\end{matrix}\right)y\\
&=&(x_1, \cdots, x_k)[(B_{12}-A_{12})y_{k+1}-A_{13}\left(\begin{array}{c}
y_{k+2}\\
 \cdots \\
 y_n\end{array}\right)]
+x_{k+1}(B_{21}-A_{21})\left(\begin{array}{c}
y_{1}\\y_{2}\\
 \cdots\\ y_k\end{array}\right)\\
&>&0.
\end{eqnarray*}
Then $\rho(D_1)>\rho(D)$. If $B_{21}-A_{21}= 0 $, then we have
  $$(x_1, \cdots, x_k)[(B_{12}-A_{12})y_{k+1}-A_{13}(y_{k+2}, \cdots, y_n)^T]> x_p(y_{k+1}-y_{k+2})|E(A_{13})|\ge 0,$$ since $A_{31}\neq 0$ and $|E(B_{12})|+|E(B_{21})|=t$. Thus we also get $\rho(D_1)>\rho(D)$. Hence we may assume that $|V(D)|=k+1$ and
  $$A(D)=\left(\begin{array}{cc}J_k-I_k & A_{12}\\ A_{21}& 0\end{array}\right)_{(k+1)\times (k+1)},$$
 where the first $p$ components of $A_{12}$ are 1 and the first $q$ components of $A_{21}$ are 1  with $  p+q=t$ and $p\le q.$  By Corollary~\ref{FrC}, $\rho(D)$ is the unique positive solution of $$\sum_{i=0}^{\infty}\frac{\mathbf{1}^T(A_{12}A_{21})^i\mathbf{1}}{r^i(r+1)^{i+1}}=1,$$
 i.e.,
 $$\frac{k}{r+1}+pq\sum_{i=1}^{\infty}\frac{(p\wedge q)^{i-1}}{r^i(r+1)^{i+1}}=1,$$
where $p\wedge q=\min\{p,q\}$.
On the other hand, $\rho(D^{\#})$ is the unique positive solution of
$$\frac{k}{r+1}+\lfloor\frac{t}{2}\rfloor\lceil\frac{t}{2}\rceil\sum_{i=1}^{\infty}
\frac{(\lfloor\frac{t}{2}\rfloor)^{i-1}}{r^i(r+1)^{i+1}}=1.$$
Hence $\rho(D)\le \rho(D^{\#})$ with equality if and only if $D=D^{\#}$.
 This completes the proof.
\end{proof}

\begin{lemma}\label{LM}
Let $e=2{k\choose 2}+t$ with $0\le t\le 2k-1$. If  $D\in \mathcal{D}^{**}(e)$ with
$$\label{E3}
A(D)=\left(\begin{matrix}J_w-I_w&A_{12}\cr A_{21}&0\end{matrix}\right).$$
Then
$$||A_{12}^T\mathbf{1}||^2+||A_{21}\mathbf{1}||^2\le pw^2+(|E(A_{12})|+|E(A_{21})|-pw)^2 , $$
$$\mathbf{1}^T(A_{12}A_{21})\mathbf{1}\le p'(w-1)w +\left\lfloor\frac{|E(A_{12})|+|E(A_{21})|-p'(2w-1)}{2}\right\rfloor\left\lceil\frac
{|E(A_{12})|+|E(A_{21})|-p'(2w-1)}{2}\right\rceil,$$
where $p=\lfloor\frac{|E(A_{12})|+|E(A_{21})|}{w}\rfloor$ and $p'=\lfloor\frac{|E(A_{12})|+|E(A_{21})|}{2w-1}\rfloor$.
\end{lemma}
\begin{proof}
Let $\alpha=(\underbrace{w,\cdots,w}_{p},|E(A_{12})|+|E(A_{21})|-pw,0,\cdots,0)^T$, then
$(\mathbf{1}^TA_{12},\mathbf{1}^TA_{21}^T)^T$ is majorized by $\alpha$. For majorization, the readers may  see \cite{MaO}. By Lemma 9 in \cite{Fr}, we have
$$||A_{12}^T\mathbf{1}||^2+||A_{21}\mathbf{1}||^2=||(\mathbf{1}^TA_{12},\mathbf{1}^TA_{21}^T)^T||^2
=(\mathbf{1}^TA_{12},\mathbf{1}^TA_{21}^T)\dot(\mathbf{1}^TA_{12},\mathbf{1}^TA_{21}^T)^T\le ||\alpha||^2.$$
Let $$|E(A_{12})|=s_1w+t_1~ \mbox{with $0\le t_1\le w-1$},$$
$$|E(A_{21})|=s_2(w-1)+t_2  ~\mbox{with $0\le t_2\le w-2$,}$$
$$|E(A_{12})|+|E(A_{21})|=p'(2w-1)+t_3~ \mbox{with $0\le t_3\le 2w-2$}.$$
Moreover, let   $\beta=(\underbrace{w,\cdots,w}_{s_1},t_1,0,\cdots,0)^T$ and $\gamma=(\underbrace{w-1,\cdots,w-1}_{s_2},t_2,0,\cdots,0)^T$. Then $\mathbf{1}^TA_{12},~\mathbf{1}^TA_{21}^T$ are majorized by $\beta,~ \gamma$, respectively. By Lemma 9 in \cite{Fr}, we have
$\mathbf{1}^T(A_{12}A_{21})\mathbf{1}\le \beta^T\gamma$.

If $s_1<s_2$, then $s_1w+t_1+s_2(w-1)+t_2\ge s_1(2w-1)$, which implies $p'\ge s_1$ and
$t_1+(w-1)\le s_1w+t_1+s_2(w-1)+t_2-p'(2w-1)\le t_3$ for $p'=s_1$.
Hence $$\beta^T\gamma\le \left\{\begin{array}{ccc}
 s_1w(w-1)+t_1(w-1)&<& p'w(w-1)+\lfloor\frac{t_3}{2}\rfloor\lceil\frac{t_3}{2}\rceil, \ \mbox {for }\ s_1<p'\\
 pw(w-1)+t_1(w-1)&\le& p'w(w-1)+\lfloor\frac{t_3}{2}\rfloor\lceil\frac{t_3}{2}\rceil, \ \mbox {for }\ s_1=p'\end{array}\right..$$

If $s_1\ge s_2$, by the same method, it is easy to prove that
$$\beta^T\gamma\le s_2w(w-1)+wt_2\le p' w(w-1)+\lfloor\frac{t_3}{2}\rfloor\lceil\frac{t_3}{2}\rceil.$$
This completes the proof of the lemma.
\end{proof}

\begin{lemma}\label{TL}
Let $e=k(k-1)+t$ with $2\le t\le 2k-1$. If $k>4t^4+4$ and $\rho(D)=\rho^{**}(e)$ for $D\in \mathcal{D^{**}}(e)$, then $D=D^\#$.
\end{lemma}
\begin{proof}
Clearly, $\rho(D)=\rho(e)\ge \rho(D^{\#})> k-1$ for $2\le t\le 2k-1$.
 By Proposition~\ref{PL1},  we can assume that
  $$
A(D)=\left(\begin{matrix}J_w-I_w&A_{12}\cr A_{21}&0\end{matrix}\right).
$$
Denote by $$e_1=|E(A_{12})|=s_1w+t_1,~\mbox{ with $0\le t_1\le w-1$ },$$
$$e_2=|E(A_{21})|=s_2(w-1)+t_2 ~\mbox{ with $0\le t_2\le w-2$}.$$
Then $w:=w(D)> k-1-\sqrt{t}$. Otherwise, by Lemma~\ref{Lu}, we have $w\le  k-1-\sqrt{t}$ and $\rho(D)\le \frac{w-1+\sqrt{(w-1)^2+2(e-w(w-1))}}{2}< k-1$.
 Further,
 $$e-w(w-1)=(k-w)(k-w-1)+2w(k-w)+t\le \sqrt{t}(\sqrt{t}+1)+2w(k-w)+t=2w(k-w)+ 2t+\sqrt{t},$$
 $$e-w(w-1)=(k-w)(k-w-1)+2w(k-w)+t\le (\sqrt{t}+1)^2+(2w-1)(k-w)+t=(2w-1)(k-w)+ 2t+2\sqrt{t}+1,$$
 which implies
$ e_1+e_2\le \min\{(2w-1)(k-w)+2t+2\sqrt{t}+1,2w(k-w)+ 2t+\sqrt{t}\}.$
By Lemma~\ref{LM}, $w > k-1-\sqrt{t}$, $k>4t^4+4$ and $ e_1+e_2\le \min\{(2w-1)(k-w)+2t+2\sqrt{t}+1,2w(k-w)+ 2t+\sqrt{t}\},$  we have
\begin{eqnarray}\label{inequality1}
&&\mathbf{1}^T(A_{12}A_{21})\mathbf{1}\le (k-w)w(w-1)+\lfloor\frac{(k-w)^2+t}{2}\rfloor\lceil\frac{(k-w)^2+t}{2}\rceil,
\end{eqnarray}
and
\begin{eqnarray}\label{inequality2}
||A_{12}^T\mathbf{1}||~||A_{21}\mathbf{1}||&\le& \frac{||A_{12}^T\mathbf{1}||^2+||A_{21}\mathbf{1}||^2}{2}\nonumber\\
&\le& (k-w)w^2+\frac{{((k-w)(k-w-1)+t)}^2}{2}\nonumber\\
&\le&(k-w)w^2+3t^2.
\end{eqnarray}
Moreover,
\begin{eqnarray*}
\sum_{i=0}^{\infty}\frac{\mathbf{1}^T(A_{12}A_{21})^i\mathbf{1}}{r^i(r+1)^{i+1}}&\le& \frac{w}{r+1}+\frac{\mathbf{1}^T(A_{12}A_{21})\mathbf{1}}{r(r+1)^2}
+\sum_{i=2}^{\infty}\frac{||A_{12}^T\mathbf{1}||~\nu(A_{12}A_{21})^{i-1}~||A_{21}\mathbf{1}||}{r^i(r+1)^{i+1}}\\
&\le&\frac{w}{r+1}+\frac{\mathbf{1}^T(A_{12}A_{21})\mathbf{1}}{r(r+1)^2}
+\frac{||A_{12}^T\mathbf{1}||~\nu(A_{12}A_{21})~||A_{21}\mathbf{1}||}{r(r+1)^{2}(r(r+1)-\nu(A_{12}A_{21}))}\\
&\le&\frac{w}{r+1}+\frac{\mathbf{1}^T(A_{12}A_{21})\mathbf{1}}{r(r+1)^2}
+\frac{||A_{12}^T\mathbf{1}||~\nu(A_{12})\nu(A_{21})~||A_{21}\mathbf{1}||}{r(r+1)^{2}(r(r+1)-\nu(A_{12})\nu(A_{21}))}\\
&\le&\frac{w}{r+1}+\frac{\mathbf{1}^T(A_{12}A_{21})\mathbf{1}}{r(r+1)^2}
+\frac{||A_{12}^T\mathbf{1}||~\sqrt{tr(A_{12})tr(A_{21})}~||A_{21}\mathbf{1}||}{r(r+1)^{2}(r(r+1)-\sqrt{tr(A_{12})tr(A_{21})})}\\
&\le& \frac{w}{r+1}+\frac{\mathbf{1}^T(A_{12}A_{21})\mathbf{1}}{r(r+1)^2}
+\frac{((k-w)w^2+3t^2)\sqrt{e_1e_2}}{r(r+1)^{2}(r(r+1)-\sqrt{e_1e_2})}.
\end{eqnarray*}
Then for $r\ge k-1$, the above inequalities are majorized by
$$f(r)=\frac{1}{r+1}\left[w+\frac{\mathbf{1}^T(A_{12}A_{21})\mathbf{1}}{k^2-k}
+\frac{((k-w)w^2+3t^2)\sqrt{e_1e_2}}{(k^2-k)(k^2-k-\sqrt{e_1e_2})}\right].$$
Since $$w+\frac{\mathbf{1}^T(A_{12}A_{21})\mathbf{1}}{k^2-k}
+\frac{((k-w)w^2+3t^2)\sqrt{e_1e_2}}{(k^2-k)(k^2-k-\sqrt{e_1e_2})}-1$$ is the solution of $f(r)=1,$
let $g(r)=\sum_{i=0}^{\infty}\frac{\mathbf{1}^T(A_{12}A_{21})^i\mathbf{1}}{r^i(r+1)^{i+1}}$. We observe  the following fact:\\
$g(r), f(r)$ are strictly decreasing on $(0,\infty)$ and $g(r)\le f(r)$ for $r$ on $[k-1,\infty)$. If $g(a)=1,\\f(b)=1$ where $0\le a,~b$, then $a\le b$. Moreover if $g(r)<f(r)$ on $[b,\infty)$, then $a< b$.\\
Therefore we have
$$\rho(D)\le w+\frac{\mathbf{1}^T(A_{12}A_{21})\mathbf{1}}{k^2-k}
+\frac{((k-w)w^2+3t^2)\sqrt{e_1e_2}}{(k^2-k)(k^2-k-\sqrt{e_1e_2})}-1.$$
If $w=k-s,~1\le s< \sqrt{t}+1$, combining with inequalities $(\ref{inequality1})$ and $(\ref{inequality2})$, we have
\begin{eqnarray*}
&&\sqrt{e_1e_2}\le \frac{e_1+e_2}{2}\le s(k-s)+t+\sqrt{t}/2,\\
&&k^2-k-\sqrt{e_1e_2}\ge k^2-k-(s(k-s)+t+\sqrt{t}/2),\\
&&\mathbf{1}^TA_{12}A_{21}\mathbf{1}\le s(k-s)(k-s-1)+(t+1)^2,\\
&&||A_{12}^T\mathbf{1}||~||A_{12}^T\mathbf{1}||\le s(k-s)^2 +3t^2.
\end{eqnarray*}
Hence
\begin{eqnarray}\label{inequality3}
\rho(D)&\le&k-s+\frac{s(k-s)(k-s-1)+(t+1)^2}{k^2-k}+\nonumber\\
&&\frac{(s(k-s)^2+3t^2)(s(k-s)+t+\sqrt{t}/2)}{(k^2-k)(k^2-k-(s(k-s)+t+\sqrt{t}/2))}-1\nonumber\\
&=&k-1+\frac{-2s^2+s^2(s+1)/k+(t+1)^2/k}{k-1}+\nonumber\\
&&\frac{(s(1-s/k)^2+3t^2/k^2)(s(1-s/k)+t/k+\sqrt{t}/{2k})}{(k-1)(1-1/k-(s(k-s)+t+\sqrt{t}/2)/k^2)}\\
&\le&k-1-\frac{2s^2}{3(k-1)}.\nonumber
\end{eqnarray}
Then $\rho(D)<k-1$, this is a contradiction. So we know that if $k>4t^4+4$, then $\rho(D)<k-1$ for $w(D)<k$. For $w(D)=k$, by Lemma~\ref{Lw}, the assertion holds. This completes the proof.
\end{proof}
{\bf Remark:} In the proof of the above theorem, we have used $k>4t^4+4$ in the inequality $(\ref{inequality3})$. Otherwise, let $s=1, ~k<(t+1)^2$, from inequality $(\ref{inequality3})$, we only can get an upper bound for $\rho(D)$, but not get $\rho(D)<k-1$. Moreover, we also can see that the formula for the upper bound is complicated. Further, we also can not get the explicit value of $\rho(D^\#)$, so it is not easy to estimate the size relation between this bound with $\rho(D^\#)$.

Now we  are ready to present the proof of the Theorem~\ref{mainth1}:\\
\begin{proof}
By Lemma~\ref{TL}, $D^{\#}$ is the only simple digraph with $e$ arcs having the maximum  spectral radius in the set $\mathcal{D^{**}}(e)$. It follows from Corollary~\ref{EL6} that apart from isolated vertices,  $D^{\#}$  is the only digraph with $e$ arcs having
the maximum  spectral radius in the set $\mathcal{D}(e)$. Hence the assertion holds.
 \end{proof}

\section{Proof of the theorem~\ref{mainth2}}
In the above section we have characterized all extremal digraphs having the maximum spectral radius in the set ${\mathcal{D}}(e)$ for  $k$  much larger than $t$.  In this section,  we characterize all extremal digraphs with special arcs number.
\begin{lemma}\label{T4}
 If $e=k(k-1)$ or $e=k(k-1)+2k-2$, then for any $D\in D^{**}(e)$, $\rho(D)\le \rho(D^\#)$ with equality if and only if $D=D^\#$.
\end{lemma}
\begin{proof}
If $e=k(k-1)$, then by Lemma~\ref{Lu},

\begin{eqnarray*}\label{E4}
\rho(D)&\le& \frac{w(D)-1+\sqrt{(w(D)-1)^2+2(e-w(D)(w(D)-1))}}{2}
\le k-1
\end{eqnarray*}
 with equality if and only if $w(D)=k$. On the other hand
  $\rho(D^{\#})=k-1$ for $D^{\#}\in \mathcal{D^{**}}(e)$.

 Now let $e=k(k-1)+2k-2$. If $w(D)\le k-1$, then by Lemma~\ref{Lu},
\begin{eqnarray*}\label{E4}
\rho(D)&\le& \frac{w(D)-1+\sqrt{(w(D)-1)^2+2(e-w(D)(w(D)-1))}}{2}\\
&\le & \frac{k-2+\sqrt{(k-2)^2+8(k-1)}}{2}=\rho(D^{\#}).
\end{eqnarray*}
Moreover, we claim $\rho(D)<\rho(D^{\#})$. In fact, if $\rho(D)=\rho(D^{\#})$, then
by the proof of Lemma~\ref{Lu},  $\nu(A_{12})=\nu(A_{21})=\sqrt{e(A_{12})}=\sqrt{e(A_{21})}=\sqrt{2(k-1)}$, which implies $rank(A_{12})=1$ and $rank(A_{21})=1$. Hence we have
$$A_{12}=\left(\begin{array}{c} J_{p\times q}\\ 0_{(k-1-p)\times q}\end{array}\right),
A_{21}=\left( J_{q\times p}\ \ 0_{q\times(k-1-p)}\right),$$
 where $pq=2(k-1)$ and $p\ge 3$. By a calculation, we have $\rho(D)<\frac{k-2+\sqrt{(k-2)^2+8(k-1)}}{2}$. It is a contradiction.

 If $w(D)=k$, by Lemma~\ref{Lw}, $\rho(D)\le \rho(D^{\#})$ with equality if and only if
 $D=D^{\#}$.
Thus the assertion holds.
\end{proof}
\begin{lemma}\label{T5}
Let $e=k(k-1)+1$. If  $D\in \mathcal{D}(e)$, then $\rho(D)\le k-1$ with equality  if and only if $D$ is,  apart from isolated vertices, the complete simple digraph $\overleftrightarrow{K_k}$ with one additional arc for $k>2$ and oriented triangle or the complete simple digraph $\overleftrightarrow{K_2}$ with one additional arc for $k=2$.
\end{lemma}
\begin{proof}
For $k=2$, there are $3$ arcs, so it is easy to check that the assertion holds. Now we  assume that $k>2$.
There does not exist a simple digraph of order $k$ and arc number $e$. So we can suppose that  $|V(D)|>k$ and $D$ contains no isolated vertex. It is sufficient to prove that if $\rho(D)=\rho(e)$, then $\rho(D)=k-1$ and $D$ contains $\overleftrightarrow{K_k}$ as its subgraph. Next we suppose $\rho(D)=\rho(e)$ and $D_1$ is a strongly connected components with $\rho(D_1)=\rho(D)$. We claim that $D_1=D$ or $D_1=\overleftrightarrow{K_k}$. If $D_1\neq \overleftrightarrow{K_k}$ and $D_1\neq D$, then there exists an edge in $E(D)$ but not in $E(D_1)$. Since $D_1\neq \overleftrightarrow{K_k}$, then $D_1$ is not a simple complete digraph, add an edge to $D_1$ getting $D_2$, by Perron-Frobenius theorem, $\rho(e)=\rho(D_1)<\rho(D_2)\le \rho(e)$, a contradiction. If $D_1=\overleftrightarrow{K_k}$, then we complete the proof. If $D_1=D$, By Proposition~\ref{T2}, we can suppose that $D\in D^{**}(e)$.
By Corollary~${3.6}$ in \cite{Zhang2002}, we can find that
$$\rho(D)< \sqrt{k(k-1)+1-(n-1)}\le k-1,$$
which is a contradiction. This completes the proof.
\end{proof}
\begin{lemma}\label{T6}
Let $e=k(k-1)+2k-1$. If  $D\in D^{**}(e)$, then $\rho(D)\le \rho(D^\#)$ with equality if and only if $D=D^\#$.
\end{lemma}
\begin{proof}
If $k=1$,  the assertion clearly holds. Now we  assume that $k>1$.
If $|V(D)|=k+1$, it is easy to check that $D=D^\#$.  If $n=|V(D)|>k+1$, by Corollary~${3.6}$ in \cite{Zhang2002},
$$\rho(D)< \sqrt{k(k-1)+2k-1-(n-1)}\le \sqrt{k^2-2}.$$
By the paragraph before Theorem~\ref{mainth1}, $A(D^\#)$ is the largest positive root of the following equation
$$\psi(\lambda)=\lambda^3-(k-2)\lambda^2-(2k-2)\lambda-(k-1)=0.$$
It is easy to see that
\begin{eqnarray*}
\psi(\sqrt{k^2-2})&=&\sqrt{k^2-2}(k^2-2k)-(k-2)(k^2-2)-(k-1) \\
&=& \frac{2(k-2)\sqrt{k^2-2}}{k+\sqrt{k^2-2}}-(k-1)\\
&<&0.\end{eqnarray*}
Hence $\rho(D)<\rho(D^{\#})$. This completes the proof.
\end{proof}

Now we  are ready to give the proof of the Theorem~\ref{mainth2}:\\
\begin{proof}
Theorem\ref{mainth2} follows from  Corollary~\ref{EL6},~Lemmas~\ref{T4},~\ref{T5} and ~\ref{T6}.
\end{proof}
\begin{corollary}
 Let $e=k(k-1)+t$ with $0\le t\le 2k-1$. Then
 for any digraph $D\in \mathcal{D}(e)$,
 $\rho(D)\le k-1+\frac{t}{2(k-1)}$.
 \end{corollary}
\begin{proof}
For $t=0,~1$, by Lemma~\ref{T4} and Lemma~\ref{T5}, the assertion holds. Next we assume that $2\le t\le 2k-1$. By Proposition~\ref{T2} and Lemma~\ref{Lu}, we have
$$\rho(D)\le \frac{k-1+\sqrt{(k-1)^2+2t}}{2}\le k-1+\frac{t}{2(k-1)}.$$
This completes the proof.
\end{proof}
Based on Theorems~\ref{mainth1}, \ref{mainth2} and the computation of spectral radius of simple digraphs in $\mathcal{D^{**}}(e)$ with the  number of  arcs  less than 75, we  may propose the  following conjecture.
\begin{conjecture}
Let $e=k(k-1)+t,~ 1<t<2k-2$. If  $D\in \mathcal{D^{**}}(e)$, then $\rho(D)\le \rho(D^\#)$ with equality if and only if $D=D^\#$.
\end{conjecture}
\textbf{Acknowledgments}
The authors are grateful to the referees for their valuable corrections and suggestions
which lead to a great improvement of this paper.

\end{document}